\documentclass[preprint,3p,12pt,authoryear]{elsarticle}

\makeatletter
\def\ps@pprintTitle{%
 \let\@oddhead\@empty
 \let\@evenhead\@empty
 \def\@oddfoot{\centerline{\thepage}}%
 \let\@evenfoot\@oddfoot}
\makeatother

\journal{}

\usepackage{amsmath,amsfonts,amssymb,amsthm,color}
\usepackage[colorlinks]{hyperref}

\usepackage{appendix} 
\usepackage{dsfont} 
\usepackage{mathtools} 
\mathtoolsset{showonlyrefs=true}
\usepackage{geometry} 
\newtheorem{theorem}{Theorem}

\newtheorem{corollary}[theorem]{Corollary}

\newtheorem*{open}{Open problem}

\newcommand{\N}{\mathbb{N}}
\newcommand{\R}{\mathbb{R}}

\newcommand{\EE}{\mathsf{E}}
\newcommand{\CE}{\mathsf{CE}}

\newcommand{\CVar}{\mathsf{CVar}}

\newcommand{\OO}{\mathcal{O}}

\newcommand{\ii}{\mathrm{i}}
\newcommand{\vv}{\sigma}
\newcommand{\sig}{\nu_{\kappa}}
\newcommand{\leqdef}{\vcentcolon=}
\newcommand{\reqdef}{=\vcentcolon}
\newcommand{\rd}{{\rm d}}

\newcommand{\fred}[1]{{\color{red} ***#1}}

\allowdisplaybreaks

\begin{document}

\begin{frontmatter}

    \title{A bridge between the circular and linear normal distributions}

    \author[a1]{Fr\'ed\'eric Ouimet}\ead{frederic.ouimet2@mcgill.ca}

    \address[a1]{Department of Mathematics and Statistics, McGill University, Montr\'eal (Qu\'ebec) Canada H3A 0B9}

    \begin{abstract}
        In this short note, we present a refined approximation for the log-ratio of the density of the von~Mises$(\mu,\kappa)$ distribution (also called the circular normal distribution) to the standard (linear) normal distribution when the concentration parameter $\kappa$ is large. Our work complements the one of \citet{doi:10.2307/2335751}, who obtained a very similar approximation along with quantile couplings, using earlier approximations by \citet{MR226762} of Cornish-Fisher type. One motivation for this note is to highlight the connection between the circular and linear normal distributions through their circular variance and (linear) variance.
    \end{abstract}


\end{frontmatter}

\section{Introduction}\label{sec:intro}

For any reals $\mu\in [0,2\pi)$ and $\kappa\in (0,\infty)$, the density function of the von~Mises distribution (also referred to as the circular normal distribution) with mean direction $\mu$ and concentration parameter~$\kappa$, henceforth denoted by $\mathrm{von~Mises}\hspace{0.2mm}(\mu,\kappa)$, is defined by
\begin{equation}\label{eq:von.Mises.density}
f_{\mu,\kappa}(\theta) = \frac{\exp\{\kappa \cos(\theta - \mu)\}}{2\pi I_0(\kappa)}, \quad \theta\in [0,2\pi),
\end{equation}
where $I_{\nu}$ denotes the modified Bessel function of the first kind of order $\nu\in \N_0 = \{0,1,\ldots\}$. The circular expectation and variance of $\Theta\sim \mathrm{von~Mises}\hspace{0.2mm}(\mu,\kappa)$ are well known to be
\begin{equation}\label{eq:noncentral.von.Mises.mean.variance}
\CE(\Theta) \equiv \mathrm{arg}\{\EE(e^{\ii \Theta})\} = \mu \quad \text{and} \quad \CVar(\Theta) \equiv 1 - \EE\{\cos(\Theta - \mu)\} = 1 - \frac{I_1(\kappa)}{I_0(\kappa)} \reqdef \sig^2,
\end{equation}
see, e.g., Section~3.5.4 of \citet{MR1828667}. As $\kappa\to \infty$, note that
\begin{equation}\label{eq:I.0.I.1}
\begin{aligned}
&I_0(\kappa) \sim \frac{e^{\kappa}}{\sqrt{2\pi \kappa}} \left\{1 + \frac{1}{8\kappa} + \frac{9}{128 \kappa^2} + \frac{75}{1024 \kappa^3} + \OO(\kappa^{-4})\right\}, \\
&I_1(\kappa) \sim \frac{e^{\kappa}}{\sqrt{2\pi \kappa}} \left\{1 - \frac{3}{8\kappa} - \frac{15}{128 \kappa^2} - \frac{105}{1024 \kappa^3} + \OO(\kappa^{-4})\right\},
\end{aligned}
\end{equation}
see, e.g., \cite[p.377]{MR0167642}, so that
\begin{equation}\label{eq:I.0.I.1.next}
\sig^2 = 1 - \frac{I_1(\kappa)}{I_0(\kappa)} = \frac{1}{2 \kappa} + \frac{1}{8 \kappa^2} + \frac{1}{8 \kappa^3} + \OO(\kappa^{-4}), \quad \kappa\to \infty.
\end{equation}
In particular, one has $\sig\to 0$ as $\kappa\to \infty$.

\newpage
The main goal of the present paper (Theorem~\ref{thm:LLT.von.Mises}) is to establish a local asymptotic expansion, as $\kappa\to \infty$, for the log-ratio of the von Mises density~\eqref{eq:von.Mises.density} to the (linear) $\mathrm{Normal}\hspace{0.2mm}(\mu,2\sig^2)$ density, viz.,
\begin{equation}\label{eq:phi.M}
\frac{1}{\sqrt{2} \sig} \phi\left(\frac{\delta_x}{\sqrt{2}}\right), \quad \text{where } \phi(z) \leqdef \frac{\exp(-z^2/2)}{\sqrt{2\pi}} \quad \text{and} \quad \delta_x \leqdef \frac{x - \mu}{\sqrt{1 - I_1(\kappa) / I_0(\kappa)}}.
\end{equation}
Notice the extra scaling factor $\sqrt{2}$ for the (linear) standard deviation in~\eqref{eq:phi.M}. This factor would not appear if the expected limiting distribution was the usual $\mathrm{Normal}\hspace{0.2mm}(\mu,\sig^2)$ distribution.

\vspace{3mm}
To explain the extra scaling factor, consider a wrapped normal distribution with mean direction~$\mu\in [0,2\pi)$ and scaling parameter $\vv\in (0,\infty)$, henceforth denoted by $\mathrm{WNormal}\hspace{0.2mm}(\mu,\vv^2)$. Its density function is given by the infinite sum
\begin{equation}\label{eq:wrapped.normal.density}
g_{\mu,v^2}(\theta) = \sum_{k=-\infty}^{\infty} \frac{1}{\vv \sqrt{2\pi}} \exp\left\{-\frac{1}{2\vv^2} (\theta - \mu + 2\pi k)^2\right\}, \quad \theta\in [0,2\pi),
\end{equation}
see, e.g., p.~50-51 of \citet{MR1828667} or Section~2.2.6 of \citet{MR1836122}; it has the same support as the von Mises distribution. The $\mathrm{WNormal}\hspace{0.2mm}(\mu,\vv^2)$ distribution can be visualized as a $\mathrm{Normal}\hspace{0.2mm}(\mu,\vv^2)$ distribution ``folded'' onto the unit circle. When the scaling parameter $\vv$ of the underlying normal distribution is small, most of the distribution's mass will be located near the mean direction $\mu$, minimizing the ``wrapping'' effect. Indeed, the expression~\eqref{eq:wrapped.normal.density} essentially sums the normal densities from each ``wrap'' around the circle. For a small $\vv$, the majority of the distribution's weight is around the mean direction $\mu$, and the contributions from terms with $k \neq 0$ (i.e., terms from subsequent wraps around the circle) will be minimal. It is well known to be a limiting distribution of the $\mathrm{von~Mises}\hspace{0.2mm}(\mu,\kappa)$ distribution when the concentration parameter $\kappa$ tends to infinity, see \citet{MR501259} or Eq.~(3.5.24) of \citet{MR1828667} for details. The circular expectation and variance of $\Theta\sim \mathrm{WNormal}\hspace{0.2mm}(\mu,\vv^2)$ are
\begin{equation}\label{eq:noncentral.von.Mises.mean.variance}
\CE(\Theta) \equiv \mathrm{arg}\{\EE(e^{\ii \Theta})\} = \mu \quad \text{and} \quad \CVar(\Theta) \equiv 1 - \EE\{\cos(\Theta - \mu)\} = 1 - \exp(-\vv^2/2).
\end{equation}
The circular variance provides a measure of the dispersion of the distribution on the circle and is derived using the so-called mean resultant length, see, e.g., Section~2.3.1 of \citet{MR1828667}. For small $\vv$, the mean resultant length is close to $1$. Specifically, when $\vv^2$ is small, one has the first-order Taylor approximation
\[
\CVar(\Theta) = 1 - \exp(-\vv^2/2) \approx \frac{\vv^2}{2} = \left(\frac{\vv}{\sqrt{2}}\right)^2,
\]
where the aforementioned extra scaling factor $\sqrt{2}$ appears again. The asymptotic relationship between the circular variance of the von Mises distribution and the corresponding limiting wrapped normal distributions is therefore
\[
\sig = \frac{\vv}{\sqrt{2}}, \quad \text{or equivalently,} \quad \vv = \sqrt{2} \sig.
\]
This explains the extra scaling factor $\sqrt{2}$ in \eqref{eq:phi.M}; it serves as a pivotal scaling relationship between the unit circle and the real line. In essence, it accounts for the additional variance caused by the subsequent wraps ($k\neq 0$) around the unit circle. This connection is also mentioned for example in Eq.~(3.4.16) of \citet{MR1828667}.

\section{Main results}

Below, we state and prove a local approximation for the ratio of the von Mises density to the linear normal density function with appropriate mean and variance. A very similar approximation was proved decades ago by \citet{doi:10.2307/2335751} along with quantile couplings, using earlier approximations by \citet{MR226762} of Cornish-Fisher type. Regarding other univariate continuous distributions, local approximations analogous to Theorem~\ref{thm:LLT.von.Mises} and Corollary~\ref{cor:cdf} below can be found for the central/noncentral chi-square distribution in \citet{MR4466042} and for the Student distribution in \citet{MR4449387}.

Throughout the paper, the notation $u = \OO(v)$ means that
\[
\limsup_{\kappa\to \infty} |u / v| \leq C \quad \text{or} \quad \limsup_{y\to \infty} |u / v| \leq C
\]
for some universal positive constant $C\in (0,\infty)$, depending on the context. Whenever $C$ might depend on some parameter, we add a subscript. For example, $u = \OO_{\eta}(v)$.

\begin{theorem}[Local approximation]\label{thm:LLT.von.Mises}
Let $\mu\in \R$ and $\eta\in (0,1)$ be given. For any $\kappa\in (0,\infty)$,
\begin{equation}\label{eq:bulk}
B_{\mu,\kappa}(\eta) \leqdef \bigg\{x\in (0,\infty) : |\delta_x| \leq \eta \, \kappa^{1/2}\bigg\},
\end{equation}
denotes the bulk of the von Mises distribution. Then, uniformly for $x\in B_{\mu,\kappa}(\eta)$, one has
\begin{equation}\label{eq:thm:LLT.von.Mises.eq.log}
\begin{aligned}
\log\left\{\frac{\sqrt{2} \sig f_{\mu,\kappa}(x)}{\phi(\widetilde{\delta}_x)}\right\}
&= \frac{1}{\kappa} \left(\frac{\widetilde{\delta}_x^4}{24} - \frac{\widetilde{\delta}_x^2}{8}\right) + \frac{1}{\kappa^2} \left(- \frac{\widetilde{\delta}_x^6}{720} + \frac{\widetilde{\delta}_x^4}{48} - \frac{\widetilde{\delta}_x^2}{8} + \frac{3}{64}\right) + \OO_{\eta}\left(\frac{1 + |\widetilde{\delta}_x|^8}{\kappa^3}\right),
\end{aligned}
\end{equation}
as $\kappa\to \infty$, where $\widetilde{\delta}_x = \delta_x / \sqrt{2}$. Moreover, uniformly for $x\in B_{\mu,\kappa}(\kappa^{-1/4} \widetilde{\eta})$ given some $\widetilde{\eta}\in (0,\infty)$, one has
\begin{equation}\label{eq:thm:LLT.von.Mises.eq}
\begin{aligned}
\frac{\sqrt{2} \sig f_{\mu,\kappa}(x)}{\phi(\widetilde{\delta}_x)}
&= 1 + \frac{1}{\kappa} \left(\frac{\widetilde{\delta}_x^4}{24} - \frac{\widetilde{\delta}_x^2}{8}\right) + \frac{1}{\kappa^2} \left(\frac{\widetilde{\delta}_x^8}{1152} - \frac{19 \widetilde{\delta}_x^6}{2880} + \frac{11 \widetilde{\delta}_x^4}{384} - \frac{\widetilde{\delta}_x^2}{8} + \frac{3}{64}\right) + \OO_{\widetilde{\eta}}\left(\frac{1 + |\widetilde{\delta}_x|^{12}}{\kappa^3}\right).
\end{aligned}
\end{equation}
\end{theorem}

\begin{open}
Extend the result of Theorem~\ref{thm:LLT.von.Mises} to the more general (spherical) von Mises-Fisher distribution. For the definition, see, e.g., Section~9.3.2 of \citet{MR1828667}.
\end{open}

Next, as a corollary, we integrate the local density approximation of Theorem~\ref{thm:LLT.von.Mises} to obtain a normal local approximation for the von Mises cumulative distribution function.

\begin{corollary}\label{cor:cdf}
Using the same notation as in Theorem~\ref{thm:LLT.von.Mises}, one has, uniformly for $x\in B_{\mu,\kappa}(\eta)$ and as $\kappa\to \infty$,
\[
\begin{aligned}
F_{\mu,\kappa}(x)
&\leqdef \int_{-\infty}^x f_{\mu,\kappa}(y) \rd y \\
&= \Phi(\widetilde{\delta}_x) + \phi(\widetilde{\delta}_x) \left\{\frac{1}{\kappa} \left(\frac{\widetilde{\delta}_x^3}{24}\right) + \frac{1}{\kappa^2} \left(\frac{\widetilde{\delta}_x^7}{1152} - \frac{\widetilde{\delta}_x^5}{1920} + \frac{5 \widetilde{\delta}_x^3}{192} - \frac{3 \widetilde{\delta}_x}{64}\right)\right\} + \OO_{\widetilde{\eta}}\left(\frac{1 + |\widetilde{\delta}_x|^{11}}{\kappa^3}\right),
\end{aligned}
\]
where
\[
\Phi(z) \leqdef \int_{-\infty}^z \phi(y) \rd y, \quad z\in \R,
\]
denotes the cumulative distribution function of the standard normal distribution.
\end{corollary}

\section{Proofs}\label{sec:proofs}

\begin{proof}[Proof of Theorem~\ref{thm:LLT.von.Mises}]
By taking the logarithm in~\eqref{eq:von.Mises.density}, one has
\begin{equation}\label{eq:LLT.beginning}
\begin{aligned}
\log\left\{\frac{\sqrt{2} \sig f_{\mu,\kappa}(x)}{\phi(\delta_x / \sqrt{2})}\right\}
&= \kappa \cos(x - \mu) - \frac{1}{2} \log(2\pi) - \log \{I_0(\kappa)\} \\[-2mm]
&\quad+ \frac{1}{2} \log 2 + \frac{1}{2} \log \sig^2 + \frac{1}{4} \delta_x^2.
\end{aligned}
\end{equation}
Since
\begin{equation}
x - \mu = \sqrt{\delta_x^2 \sig^2} = \sqrt{\delta_x^2 \left\{\frac{1}{2 \kappa} + \frac{1}{8 \kappa^2} + \frac{1}{8 \kappa^3} + \OO(\kappa^{-4})\right\}},
\end{equation}
(recall \eqref{eq:I.0.I.1.next}) and
\begin{equation}
\cos(\sqrt{y}) = 1 - \frac{y}{2} + \frac{y^2}{24} - \frac{y^3}{720} + \OO(y^4), \quad |y| < 1,
\end{equation}
we know that
\begin{equation}\label{eq:cos.expansion}
\begin{aligned}
\kappa \cos(x - \mu)
&= \kappa \left[1 - \frac{\delta_x^2}{2} \left\{\frac{1}{2 \kappa} + \frac{1}{8 \kappa^2} + \frac{1}{8 \kappa^3} + \OO(\kappa^{-4})\right\} + \frac{\delta_x^4}{24} \left\{\frac{1}{2 \kappa} + \frac{1}{8 \kappa^2} + \frac{1}{8 \kappa^3} + \OO(\kappa^{-4})\right\}^2 \right. \\[1mm]
&\qquad \left.- \frac{\delta_x^6}{720} \left\{\frac{1}{2 \kappa} + \frac{1}{8 \kappa^2} + \frac{1}{8 \kappa^3} + \OO(\kappa^{-4})\right\}^3 + \OO_{\eta}\left(\frac{1 + |\delta_x|^8}{\kappa^4}\right)\right] \\
&= \kappa - \frac{1}{4} \delta_x^2 + \frac{1}{\kappa} \left(\frac{\delta_x^4}{96} - \frac{\delta_x^2}{16}\right) + \frac{1}{\kappa^2} \left(- \frac{\delta_x^6}{5760} + \frac{\delta_x^4}{192} - \frac{\delta_x^2}{16}\right) + \OO_{\eta}\left(\frac{1 + |\delta_x|^8}{\kappa^3}\right).
\end{aligned}
\end{equation}
By~\eqref{eq:I.0.I.1},~\eqref{eq:I.0.I.1.next} and the Taylor expansion,
\begin{equation}
\log(1 + y) = y - \frac{y^2}{2} + \frac{y^3}{3} + \OO(y^4), \quad |y| < 1,
\end{equation}
one has
\begin{equation}
\begin{aligned}
- \log \{I_0(\kappa)\} &= - \kappa + \frac{1}{2} \log (2\pi \kappa) - \frac{1}{8 \kappa} - \frac{1}{16 \kappa^2} - \frac{25}{384 \kappa^3} + \OO(\kappa^{-4}), \\[2mm]
\frac{1}{2} \log \sig^2 &= - \frac{1}{2} \log(2 \kappa) + \frac{1}{8 \kappa} + \frac{7}{64 \kappa^2} - \frac{11}{384 \kappa^3} + \OO(\kappa^{-4}),
\end{aligned}
\end{equation}
so that
\begin{equation}\label{eq:expansion.simplify}
- \frac{1}{2} \log(2\pi) - \log \{I_0(\kappa)\} + \frac{1}{2} \log 2 + \frac{1}{2} \log \sig^2 = - \kappa + \frac{3}{64 \kappa^2} - \frac{3}{32 \kappa^3} + \OO(\kappa^{-4}).
\end{equation}
By putting~\eqref{eq:cos.expansion} and~\eqref{eq:expansion.simplify} together in~\eqref{eq:LLT.beginning}, we obtain
\[
\log\left\{\frac{\sqrt{2} \sig f_{\mu,\kappa}(x)}{\phi(\delta_x / \sqrt{2})}\right\} = \frac{1}{\kappa} \left(\frac{\delta_x^4}{96} - \frac{\delta_x^2}{16}\right) + \frac{1}{\kappa^2} \left(- \frac{\delta_x^6}{5760} + \frac{\delta_x^4}{192} - \frac{\delta_x^2}{16} + \frac{3}{64}\right) + \OO_{\eta}\left(\frac{1 + |\delta_x|^8}{\kappa^3}\right),
\]
which proves~\eqref{eq:thm:LLT.von.Mises.eq.log}.

To obtain~\eqref{eq:thm:LLT.von.Mises.eq} and conclude the proof, we take the exponential on both sides of the last equation with $\eta = \kappa^{-1/4} \widetilde{\eta}$, and we expand the right-hand side with
\begin{equation}\label{eq:Taylor.exponential}
e^y = 1 + y + \frac{y^2}{2} + \OO(e^U y^3), \quad \text{for } -\infty < y \leq U < \infty.
\end{equation}
We get
\[
\begin{aligned}
\frac{\sqrt{2} \sig f_{\mu,\kappa}(x)}{\phi(\delta_x / \sqrt{2})}
&= 1 + \frac{1}{\kappa} \left(\frac{\delta_x^4}{96} - \frac{\delta_x^2}{16}\right) + \frac{1}{\kappa^2} \left(- \frac{\delta_x^6}{5760} + \frac{\delta_x^4}{192} - \frac{\delta_x^2}{16} + \frac{3}{64}\right) \\
&\qquad+ \frac{1}{2\kappa^2} \left(\frac{\delta_x^4}{96} - \frac{\delta_x^2}{16}\right)^2 + \OO_{\widetilde{\eta}}\left(\frac{1 + |\delta_x|^{12}}{\kappa^3}\right) \\
&= 1 + \frac{1}{\kappa} \left(\frac{\delta_x^4}{96} - \frac{\delta_x^2}{16}\right) + \frac{1}{\kappa^2} \left(\frac{\delta_x^8}{18432} - \frac{19 \delta_x^6}{23040} + \frac{11 \delta_x^4}{1536} - \frac{\delta_x^2}{16} + \frac{3}{64}\right) + \OO_{\widetilde{\eta}}\left(\frac{1 + |\delta_x|^{12}}{\kappa^3}\right),
\end{aligned}
\]
which yields the claim~\eqref{eq:thm:LLT.von.Mises.eq} after some rearrangements. This concludes the proof.
\end{proof}

\begin{proof}[Proof of Corollary~\ref{cor:cdf}]
By integrating from $x$ to $\infty$ the result of Theorem~\ref{thm:LLT.von.Mises}, one gets
\[
\begin{aligned}
\int_x^{\infty} f_{\mu,\kappa}(y) \rd y
&= \Psi(\widetilde{\delta}_x) + \frac{1}{\kappa} \left(\frac{\Psi_4(\widetilde{\delta}_x)}{24} - \frac{\Psi_2(\widetilde{\delta}_x)}{8}\right) \\
&\quad+ \frac{1}{\kappa^2} \left(\frac{\Psi_8(\widetilde{\delta}_x)}{1152} - \frac{19 \Psi_6(\widetilde{\delta}_x)}{2880} + \frac{11 \Psi_4(\widetilde{\delta}_x)}{384}- \frac{\Psi_2(\widetilde{\delta}_x)}{8} + \frac{3}{64}\right) + \OO_{\widetilde{\eta}}\left(\frac{1 + |\widetilde{\delta}_x|^{11}}{\kappa^3}\right),
\end{aligned}
\]
where
\[
\Psi(z) = \int_z^{\infty} \phi(y) \rd y, \quad \Psi_j(z) = \int_z^{\infty} y^j \phi(y) \rd y, \quad z\in \R, ~j\in \N.
\]
Since
\[
\begin{aligned}
\Psi_2(z) &= z \phi(z) + \Psi(z), \\
\Psi_4(z) &= (3z + z^3) \phi(z) + 3 \Psi(z), \\
\Psi_6(z) &= (15z + 5z^3 + z^5) \phi(z) + 15 \Psi(z), \\
\Psi_8(z) &= (105z + 35z^3 + 7z^5 + z^7) \phi(z) + 105 \Psi(z),
\end{aligned}
\]
one deduces
\[
\begin{aligned}
F_{\mu,\kappa}(x) = \Phi(\widetilde{\delta}_x) + \phi(\widetilde{\delta}_x) \left\{\frac{1}{\kappa} \left(\frac{\widetilde{\delta}_x^3}{24}\right) + \frac{1}{\kappa^2} \left(\frac{\widetilde{\delta}_x^7}{1152} - \frac{\widetilde{\delta}_x^5}{1920} + \frac{5 \widetilde{\delta}_x^3}{192} - \frac{3 \widetilde{\delta}_x}{64}\right)\right\} + \OO_{\widetilde{\eta}}\left(\frac{1 + |\widetilde{\delta}_x|^{11}}{\kappa^3}\right).
\end{aligned}
\]
This concludes the proof.
\end{proof}

%
%

\addcontentsline{toc}{chapter}{References}

\bibliographystyle{authordate1}
\bibliography{Ouimet_2023_LLT_von_Mises_bib}

\begin{thebibliography}{}

\bibitem[\protect\citename{Abramowitz \& Stegun, }1964]{MR0167642}
Abramowitz, M., \& Stegun, I.~A. 1964.
\newblock {\em Handbook of {M}athematical {F}unctions {W}ith {F}ormulas,
  {G}raphs, and {M}athematical {T}ables}.
\newblock National Bureau of Standards Applied Mathematics Series, vol. 55.
\newblock For sale by the Superintendent of Documents, U.S. Government Printing
  Office, Washington, D.C.
\newblock \href{http://www.ams.org/mathscinet-getitem?mr=MR0167642}{MR0167642}.

\bibitem[\protect\citename{Hill, }1976]{doi:10.2307/2335751}
Hill, G.~W. 1976.
\newblock New approximations to the von {M}ises distribution.
\newblock {\em Biometrika}, {\bf 63}(3), 673--676.
\newblock \href{https://www.doi.org/10.2307/2335751}{doi:10.2307/2335751}.

\bibitem[\protect\citename{Hill \& Davis, }1968]{MR226762}
Hill, G.~W., \& Davis, A.~W. 1968.
\newblock Generalized asymptotic expansions of {C}ornish-{F}isher type.
\newblock {\em Ann. Math. Statist.}, {\bf 39}, 1264--1273.
\newblock \href{http://www.ams.org/mathscinet-getitem?mr=MR226762}{MR226762}.

\bibitem[\protect\citename{Jammalamadaka \& SenGupta, }2001]{MR1836122}
Jammalamadaka, S.~R., \& SenGupta, A. 2001.
\newblock {\em Topics in {C}ircular {S}tatistics}.
\newblock Series on Multivariate Analysis, vol. 5.
\newblock World Scientific Publishing Co., Inc., River Edge, NJ.
\newblock With 1 IBM-PC floppy disk (3.5 inch; HD).

\bibitem[\protect\citename{Kent, }1978]{MR501259}
Kent, J. 1978.
\newblock Limiting behaviour of the von {M}ises-{F}isher distribution.
\newblock {\em Math. Proc. Cambridge Philos. Soc.}, {\bf 84}(3), 531--536.
\newblock \href{http://www.ams.org/mathscinet-getitem?mr=MR501259}{MR501259}.

\bibitem[\protect\citename{Mardia \& Jupp, }2000]{MR1828667}
Mardia, K.~V., \& Jupp, P.~E. 2000.
\newblock {\em Directional {S}tatistics}.
\newblock Wiley Series in Probability and Statistics.
\newblock John Wiley \& Sons, Ltd., Chichester.
\newblock \href{http://www.ams.org/mathscinet-getitem?mr=MR1828667}{MR1828667}.

\bibitem[\protect\citename{Ouimet, }2022a]{MR4466042}
Ouimet, F. 2022a.
\newblock Refined normal approximations for the central and noncentral
  chi-square distributions and some applications.
\newblock {\em Statistics}, {\bf 56}(4), 935--956.
\newblock \href{http://www.ams.org/mathscinet-getitem?mr=MR4466042}{MR4466042}.

\bibitem[\protect\citename{Ouimet, }2022b]{MR4449387}
Ouimet, F. 2022b.
\newblock Refined normal approximations for the {S}tudent distribution.
\newblock {\em J. Class. Anal.}, {\bf 20}(1), 23--33.
\newblock \href{http://www.ams.org/mathscinet-getitem?mr=MR4449387}{MR4449387}.

\end{thebibliography}

\end{document}